\documentclass[runningheads]{llncs}
\usepackage{graphicx}
\usepackage{latexsym, amsfonts, amsmath, amssymb}
\usepackage[utf8]{inputenc}
\usepackage[noadjust]{cite}
\usepackage{enumitem}
\usepackage{color}
\usepackage{mathtools}
\usepackage{bm}
\usepackage{subcaption}
\usepackage{bbm}
\usepackage{ifdraft}
\usepackage{arydshln}
\usepackage{mdframed}
\usepackage{xcolor}
\usepackage{tikz}
\usepackage{orcidlink}
\usepackage{marvosym}

\usepackage{hyperref}
\hypersetup{
	final,
    colorlinks=true,
    linkcolor=blue,
    filecolor=magenta,
    urlcolor=cyan,
}

\makeatletter
\renewcommand*{\@fnsymbol}[1]{\ensuremath{\ifcase#1\or *\or \dagger\or \ddagger\or
   \mathsection\or \mathparagraph\or \|\or **\or \dagger\dagger
   \or \ddagger\ddagger \else\@ctrerr\fi}}
\makeatother

\newif\ifrevised

\revisedtrue

\renewcommand{\underbar}[1]{\mkern 1.5mu\underline{\mkern-1.5mu#1\mkern-1.5mu}\mkern 1.5mu}

\providecommand*{\N}[1]{\left\|{#1}\right\|} 
\providecommand*{\Nnormal}[1]{\|{#1}\|} 

\DeclarePairedDelimiter\abs{\lvert}{\rvert}%
\renewcommand{\abs}[1]{\left|{#1}\right|}

\newcommand{\divergence}{\textrm{div}}

\newcommand{\diag}{\textrm{diag}}

\DeclareMathOperator*{\relu}{ReLU}
\providecommand*{\Var}[1]{\operatorname{Var}\left({#1}\right)}   

\newcommand{\globmin}{v^*}
\newcommand{\minobj}{\underbar \CE}

\newcommand{\indivmeasure}[0]{\varrho} 

\newcommand{\empmeasurenoarg}[0]{\widehat\rho^N}
\newcommand{\empmeasure}[1]{\widehat\rho_{#1}^N}

\newcommand{\omegaa}[0]{\omega_{\alpha}}

\newcommand{\conspointnoarg}{v_{\alpha}}
\newcommand{\conspoint}[1]{v_{\alpha}({#1})}

\renewcommand{\Im}{\operatorname{Im}}             
\providecommand{\argmin}{\operatorname*{arg\,min}}  
\providecommand{\Id}{\mathrm{Id}}                     
\providecommand{\Supp}{\operatorname{supp}}                            
\providecommand{\supp}{\Supp}

\providecommand{\bbR}{\mathbb{R}}
\providecommand{\bbE}{\mathbb{E}}
\providecommand{\CC}{{\cal C}}
\providecommand{\CP}{{\cal P}}
\providecommand{\CE}{{\cal E}}
\providecommand{\CV}{{\cal V}}

\providecommand{\CN}{{\cal N}}

\begin{document}
\title{Convergence of Anisotropic Consensus-Based Optimization in Mean-Field Law}
\titlerunning{Convergence of Anisotropic Consensus-Based Optimization}

\def\doi#1{\href{https://doi.org/\detokenize{#1}}{\url{https://doi.org/\detokenize{#1}}}}

\author{
Massimo Fornasier\inst{1,2}\thanks{Email: \texttt{massimo.fornasier@ma.tum.de}}\orcidlink{0000-0001-6854-2445}
\and
Timo Klock\inst{3}\thanks{Email: \texttt{timo@simula.no}}\orcidlink{0000-0002-0122-3017} 
\and
Konstantin Riedl\inst{1}\thanks{Email: \texttt{konstantin.riedl@ma.tum.de}}$^{(\text{\Letter})}$\orcidlink{0000-0002-2206-4334}
}
\authorrunning{M.\@~Fornasier, T.\@~Klock, and K.\@~Riedl}

\institute{
	Technical University of Munich, Department of Mathematics, Munich, Germany \and
	Munich Data Science Institute, Munich, Germany \and
	Simula Research Laboratory, Department of Numerical Analysis and Scientific Computing, Oslo, Norway
	}

\maketitle
\begin{abstract}
In this paper we study anisotropic consensus-based optimization~(CBO), a population-based metaheuristic derivative-free optimization method capable of globally minimizing nonconvex and nonsmooth functions in high dimensions.
CBO is based on stochastic swarm intelligence, and inspired by consensus dynamics and opinion formation.
Compared to other metaheuristic algorithms like Particle Swarm Optimization, CBO is of a simpler nature and therefore more amenable to theoretical analysis.
By adapting a recently established proof technique, we show that anisotropic CBO converges globally with a dimension-independent rate for a rich class of objective functions under minimal assumptions on the initialization of the method.
Moreover, the proof technique reveals that CBO performs a convexification of the optimization problem as the number of particles goes to infinity, thus providing an insight into the internal CBO mechanisms responsible for the success of the method.
To motivate anisotropic CBO from a practical perspective, we further test the method on a complicated high-dimensional benchmark problem, which is well understood in the machine learning literature.

\end{abstract}

\section{Introduction} \label{sec:introduction}
Several problems arising throughout all quantitative disciplines are concerned with the global unconstrained optimization of a problem-dependent objective function $\CE : \bbR^d\rightarrow \bbR$ and the search for the associated minimizing argument
\begin{equation*}
	\globmin = \argmin_{v\in \bbR^d}\CE(v),
\end{equation*}
which is assumed to exist and be unique in what follows.
Because of nowadays data deluge such optimization problems are usually high-dimensional.
In machine learning, for instance, one is interested in finding the optimal parameters of a neural network~(NN) to accomplish various tasks, such as clustering, classification, and regression.
The availability of huge amounts of training data for various real-world applications allows practitioners to work with models involving a large number of trainable parameters aiming for a high expressivity and accuracy of the trained model.
This makes the resulting optimization process a high-dimensional problem.
Since typical model architectures consist of many layers with a large amount of neurons, and include nonlinear and potentially nonsmooth activation functions, the training process is in general a high-dimensional nonconvex optimization problem and therefore a particularly hard task.

Metaheuristics have a long history as state-of-the-art methods when it comes to tackling hard optimization problems.
Inspired by self-organization and collective behavior in nature or human society, such as the swarming of flocks of birds or schools of fish~\cite{carrillo2010swarming}, or opinion formation~\cite{toscani2006opinion},
they orchestrate an interplay between locally confined procedures and global strategies, randomness and deterministic decisions, to ensure a robust search for the global minimizer.
Some prominent examples are Random Search~\cite{rastrigin1963convergence}, Evolutionary Programming~\cite{fogel2006evolutionary}, Genetic Algorithms~\cite{holland1992adaptation}, Ant Colony Optimization~\cite{dorigo2005ant}, Particle Swarm Optimization~\cite{kennedy1995particle} and Simulated Annealing~\cite{aarts1989simulated}.

CBO follows those guiding principles, but is of much simpler nature and more amenable to theoretical analysis.
The method uses $N$ particles $V^1,\ldots,V^N$, which are initialized independently according to some law~$\rho_0\in\CP(\bbR^d)$, to explore the domain and to form a global consensus about the minimizer~$\globmin$ as time passes.
For parameters $\alpha,\lambda,\sigma > 0$ the dynamics of each particle is given by
\begin{equation} \label{eq:dyn_micro}
	dV_t^i =
	-\lambda\left(V_t^i - \conspoint{\empmeasure{t}}\right)dt
	+\sigma D\!\left(V_t^i - \conspoint{\empmeasure{t}}\right) dB^i_t,
\end{equation}
where $\empmeasure{t}$ denotes the empirical measure of the particles.
The first term in~\eqref{eq:dyn_micro} is a drift term dragging the respective particle towards the momentaneous consensus point, a weighted average of the particles' positions, computed as
\begin{align*} 
	\conspoint{\empmeasure{t}}
	:= \int v \frac{\omegaa(v)}{\N{\omegaa}_{L_1(\empmeasure{t})}}\,d\empmeasure{t}(v),
	\quad
	\textrm{with}\quad \omegaa(v) := \exp(-\alpha \CE(v))
\end{align*}
and motivated by the fact that $\conspoint{\empmeasure{t}}\approx\argmin_{i=1,\dots,N}\CE(V_t^i)$ for $\alpha\gg1$ if the $\argmin$ is unique.
To feature the exploration of the energy landscape of~$\CE$, the second term in~\eqref{eq:dyn_micro} is a diffusion injecting randomness into the dynamics through independent standard Brownian motions $((B_t^i)_{t\geq 0})_{i=1,\dots,N}$.
The two commonly studied diffusion types are isotropic~\cite{pinnau2017consensus,carrillo2018analytical,fornasier2021consensus} and anisotropic~\cite{carrillo2019consensus} diffusion with
\begin{equation*} 
D\!\left(V_t^i-\conspoint{\empmeasure{t}}\right) =
\begin{cases}
	\N{V_t^i-\conspoint{\empmeasure{t}}}_2 \Id ,& \textrm{ for isotropic diffusion,}\\
	\diag\left(V_t^i-\conspoint{\empmeasure{t}}\right)\!,& \textrm{ for anisotropic diffusion},
\end{cases}
\end{equation*}
where $\Id\in\bbR^{d\times d}$ is the identity matrix and $\diag:\bbR^d\rightarrow\bbR^{d\times d}$ the operator mapping a vector onto a diagonal matrix with the vector as its diagonal.
The term's scaling encourages in particular particles far from $\conspoint{\empmeasure{t}}$ to explore larger regions.
The coordinate-dependent scaling of anisotropic diffusion has proven to be particularly beneficial for high-dimensional optimization problems by yielding dimension-independent convergence rates (see Figure~\ref{fig:Vindifferentd}) and therefore improving both computational complexity and success probability of the algorithm~\cite{carrillo2019consensus,fornasier2021anisotropic}.

A theoretical convergence analysis of the CBO dynamics is possible either on the microscopic level~\eqref{eq:dyn_micro} or by analyzing the macroscopic behavior of the particle density through a mean-field limit.
In the large particle limit a particle is not influenced by individual particles but only by the average behavior of all particles.
As shown in~\cite{huang2021MFLCBO}, the empirical random particle measure $\empmeasurenoarg$ converges in law to the deterministic particle density~$\rho\in\CC([0,T],\CP(\bbR^d))$, which weakly (see Definition~\ref{def:fokker_planck_anisotropic_weak_sense}) satisfies the non-linear Fokker-Planck equation
\begin{align} \label{eq:fokker_planck_anisotropic}
	\partial_t\rho_t
	= \lambda\divergence \big(\!\left(v - \conspoint{\rho_t}\right)\rho_t\big)
	+ \frac{\sigma^2}{2}\sum_{k=1}^d\partial_{kk}\left(D\!\left(v-\conspoint{\rho_t}\right)_{kk}^2\rho_t\right).
\end{align}
A quantitative analysis of the convergence rate remains, on non-compact domains, an open problem, see, e.g., \cite[Remark~2]{fornasier2021consensus}.
Analyzing a mean-field limit such as~\eqref{eq:fokker_planck_anisotropic} allows for establishing strong qualitative theoretical guarantees about CBO methods, paving the way to understand the internal mechanisms at play.

\subsubsection{Prior Arts.}
The original CBO work~\cite{pinnau2017consensus} proposes the dynamics~\eqref{eq:dyn_micro} with isotropic diffusion, which is analyzed in the mean-field sense in~\cite{carrillo2018analytical}.
Under a stringent well-preparedness condition about~$\rho_0$ and $\CC^2$ regularity of $\CE$ the authors show consensus formation of the particles at some $\widetilde{v}$ close to $\globmin$ by first establishing exponential decay of the variance~$\Var{\rho_t}$ and consecutively showing $\widetilde{v}\approx\globmin$ as a consequence of the Laplace principle~\cite{miller2006applied}.
This analysis is extended to the anisotropic case in~\cite{carrillo2019consensus}.

Motivated by the surprising phenomenon that, on average, individual particles of the CBO dynamics follow the gradient flow of $v\!\mapsto\!\N{v-\globmin}_2^2$, see~\cite[Figure~1b]{fornasier2021consensus}, the authors of~\cite{fornasier2021consensus} develop a novel proof technique for showing global convergence of isotropic CBO in mean-field law to $\globmin$ under minimal assumptions.
Following \cite[Definition~1]{fornasier2021consensus}, we speak of convergence in mean-field law to $\globmin$ for the interacting particle dynamics~\eqref{eq:dyn_micro}, if the solution~$\rho_t$ to the associated mean-field limit dynamics~\eqref{eq:fokker_planck_anisotropic} converges narrowly to the Dirac delta $\delta_{\globmin}$ at $\globmin$ for $t\rightarrow \infty$.
The proof is based on showing an exponential decay of the energy functional~\mbox{$\CV:\CP(\bbR^d)\!\rightarrow\!\bbR_{\geq0}$}, given by
\begin{equation} \label{eq:energy_functional}
	\CV(\rho_t) := \frac{1}{2} \int \N{v-\globmin}_2^2 d\rho_t(v).
\end{equation}
This simultaneously ensures consensus formation and convergence of $\rho_t$ to $\delta_{\globmin}$.

\subsubsection{Contribution. }
In view of the effectiveness and efficiency of CBO methods with anisotropic diffusion for high-dimensional optimization problems, a thorough understanding is of considerable interest.
\begin{figure}[!ht]
\centering
\subcaptionbox{\footnotesize\label{fig:objective1d} The Rastrigin function in one coordinate direction}{\includegraphics[height=4.6cm, width=0.36\textwidth, trim=120 228 124 244,clip]{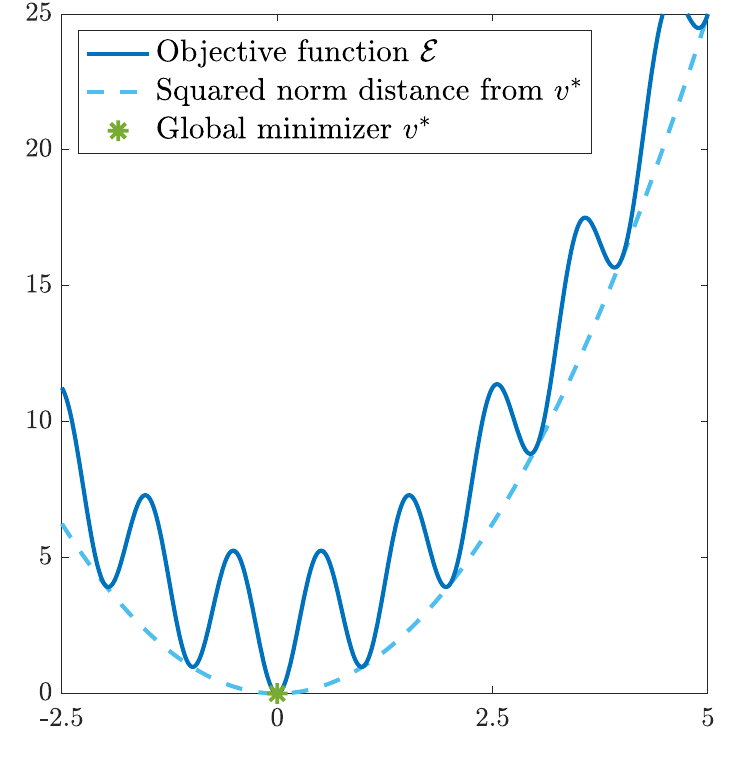}}
\hspace{3em}
\subcaptionbox{\label{fig:comparison} Evolution of $\CV(\empmeasure{t})$ for isotropic and anisotropic CBO for different dimensions}{\includegraphics[height=4.6cm, width=0.516\textwidth, trim=48 228 40 244,clip]{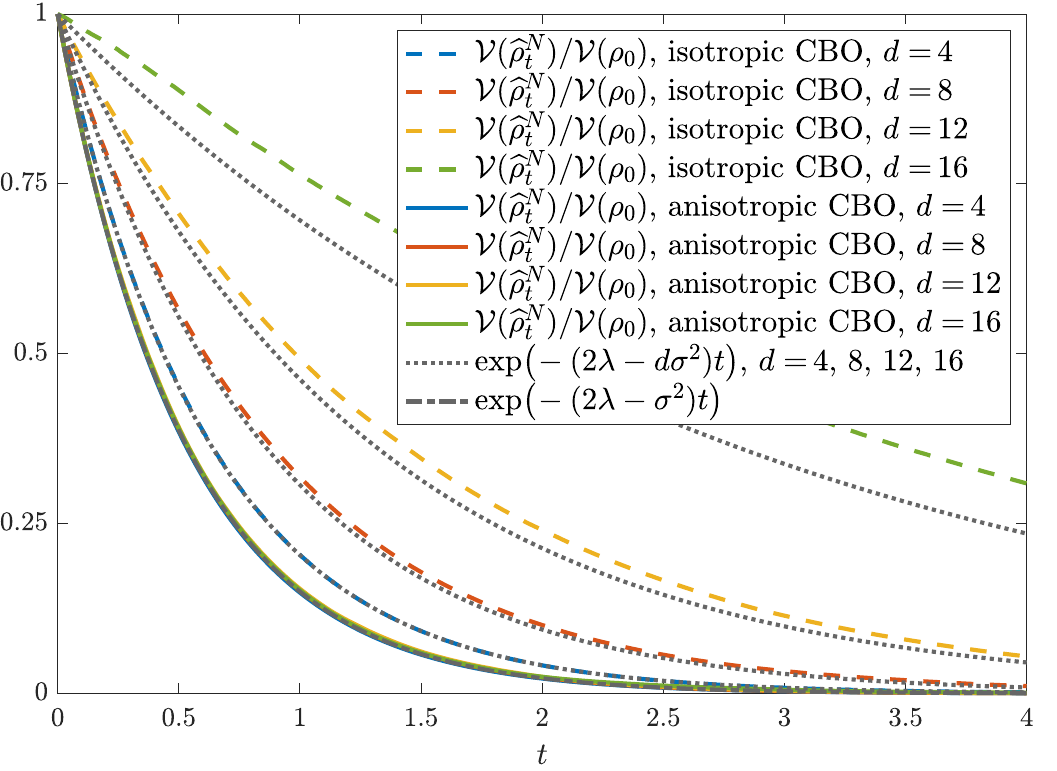}}
\caption{A demonstration of the benefit of using anisotropic diffusion in CBO.
For the Rastrigin function~$\CE(v)\!=\!\sum_{k=1}^d \!v_k^2\!+\!\frac{5}{2}(1\!-\!\cos(2\pi v_k))$ with~$\globmin \!=\! 0$ and spurious local minima (see {(a)}), we evolve the discretized system of isotropic and anisotropic CBO using $N = 320000$ particles, discrete time step size~$\Delta t = 0.01$ and \mbox{$\alpha = 10^{15}$}, $\lambda = 1$, and $\sigma = 0.32$ for different dimensions~$d\in\{4,8,12,16\}$.\\
We observe in {(b)} that the convergence rate of the energy functional~$\CV(\empmeasure{t})$ for isotropic CBO (dashed lines) is affected by the ambient dimension~$d$, whereas anisotropic CBO (solid lines) converges independently from $d$ with rate $(2\lambda-\sigma^2)$.} \vspace{-1em}
\label{fig:Vindifferentd}
\end{figure}
As we illustrate in Figure~\ref{fig:Vindifferentd}, anisotropic CBO~\cite{carrillo2019consensus} converges with a dimension-independent rate as opposed to isotropic CBO~\cite{pinnau2017consensus,carrillo2018analytical,fornasier2021consensus}, 
making it a particularly interesting choice for problems in high-dimensional spaces, e.g., from signal processing and machine learning applications.
In this work we extend the analysis of \cite{fornasier2021consensus} from isotropic CBO to CBO with anisotropic diffusion.
More precisely, we show global convergence of the anisotropic CBO dynamics in mean-field law to the global minimizer~$\globmin$ under minimal assumptions about the initial measure~$\rho_0$ and for a rich class of objectives~$\CE$.
Furthermore, utilizing some tweaks in the implementation of anisotropic CBO, such as a random mini-batch idea and a cooling strategy of the parameters as proposed in~\cite{carrillo2019consensus,fornasier2020consensus_sphere_convergence},
we show that CBO performs well, in fact, almost on par with state-of-the-art gradient-based methods, on a long-studied machine learning benchmark in $2000$ dimensions, despite using just $100$ particles and no gradient information.
This encourages the use and further investigation of CBO as a training algorithm for challenging machine learning tasks.

\subsubsection{Organization.} \label{subsec:organization}
In Section~\ref{sec:convergence} we first recall details about the well-posedness of the mean-field dynamics~\eqref{eq:fokker_planck_anisotropic} in the case of anisotropic diffusion before we present the main theoretical result about the convergence of anisotropic CBO in mean-field law.
The proof follows in Section~\ref{sec:proof_main_theorem_anisotropic}.
Section~\ref{sec:numerics} illustrates the practicability of the method on a benchmark problem and Section~\ref{sec:conclusion} concludes the paper.

For the sake of reproducible research, in the GitHub repository \url{https://github.com/KonstantinRiedl/CBOGlobalConvergenceAnalysis} we provide the Matlab code implementing the CBO algorithm used in this work.

\subsubsection{Notation.}
$B^\infty_{r}(u)$ is a closed $\ell^\infty$ ball in $\bbR^d$ with center~$u$ and radius~$r$.
For the space of continuous functions $f:X\rightarrow Y$ we write $\CC(X,Y)$, with $X\subset\bbR^n$, $n\in\mathbb{N}$, and a suitable topological space~$Y$.
For $X\subset\bbR^n$ open and for $Y=\bbR^m$, $m\in\mathbb{N}$, the function space $\CC^k_{c}(X,Y)$ contains functions $f\in\CC(X,Y)$ that are $k$-times continuously differentiable and compactly supported.
$Y$ is omitted if $Y=\bbR$.

The objects of study are laws of stochastic processes, $\rho\in\CC([0,T],\CP(\bbR^d))$, where $\CP(\bbR^d)$ contains all Borel probability measures over $\bbR^d$.
$\rho_t\in\CP(\bbR^d)$ is a snapshot of such law at time~$t$ and $\indivmeasure$ some fixed distribution.
Measures~$\indivmeasure \in \CP(\bbR^d)$ with finite $p$-th moment are collected in $\CP_p(\bbR^d)$.
For any $1\leq p<\infty$, $W_p$ denotes the \mbox{Wasserstein-$p$} distance.
$\bbE(\indivmeasure)$ is the expectation of a probability measure $\indivmeasure$.

\section{Global Convergence in Mean-Field Law} \label{sec:convergence}

In this section we first recite a well-posedness result about the Fokker-Planck equation~\eqref{eq:fokker_planck_anisotropic} and then present the main result about global convergence.

\subsection{Definition of Weak Solutions and Well-Posedness} \label{subsec:definition of weak solutions}

We begin by defining weak solutions of the Fokker-Planck equation~\eqref{eq:fokker_planck_anisotropic}.

\begin{definition} \label{def:fokker_planck_anisotropic_weak_sense}
	Let $\rho_0 \in \CP(\bbR^d)$, $T > 0$.
	We say $\rho\in\CC([0,T],\CP(\bbR^d))$ satisfies the Fokker-Planck equation~\eqref{eq:fokker_planck_anisotropic} with initial condition $\rho_0$ in the weak sense in the time interval $[0,T]$, if we have for all $\phi \in \CC_c^{\infty}(\bbR^d)$ and all $t \in (0,T)$
	\begin{equation} \label{eq:weak_solution_identity_anisotropic}
	\begin{aligned}
		\frac{d}{dt}\int \phi(v) \,d\rho_t(v) =
		&- \lambda\int \sum_{k=1}^d (v - \conspoint{\rho_t})_k \partial_k \phi(v) \, d\rho_t(v)\\
		&+ \frac{\sigma^2}{2} \int \sum_{k=1}^d D\!\left(v-\conspoint{\rho_t}\right)_{kk}^2  \partial^2_{kk} \phi(v) \,d\rho_t(v)
	\end{aligned}
	\end{equation}
	and $\lim_{t\rightarrow 0}\rho_t = \rho_0$ pointwise.
\end{definition}

In what follows the case of CBO with anisotropic diffusion is considered, i.e.,~$D\!\left(v-\conspoint{\cdot}\right) = \diag\left(v-\conspoint{\cdot}\right)$ in Equations~\eqref{eq:dyn_micro}, \eqref{eq:fokker_planck_anisotropic} and~\eqref{eq:weak_solution_identity_anisotropic}.

Analogously to the well-posedness results \cite[Theorems~3.1, 3.2]{carrillo2018analytical} for CBO with isotropic diffusion, we can obtain well-posedness of~\eqref{eq:fokker_planck_anisotropic} for anisotropic CBO.

\begin{theorem} \label{thm:well-posedness_FP_anisotropic}
	Let $T > 0$, $\rho_0 \in \CP_4(\bbR^d)$ and consider $\CE : \bbR^d\rightarrow \bbR$ with $\minobj:=\CE(\globmin) > -\infty$, which, for some constants $C_1,C_2 > 0$, satisfies
	\begin{align*}
		\abs{\CE(v)-\CE(w)} &\leq C_1(\N{v}_2 + \N{w}_2)\N{v-w}_2,\quad \textrm{for all } v,w \in \bbR^d,\\
		\CE(v) - \minobj &\leq C_2 \big(1+\N{v}_2^2\big),\quad \textrm{for all } v \in \bbR^d.
	\end{align*}
	If in addition, either $\sup_{v \in \bbR^d}\CE(v) < \infty$, or, for some $C_3,C_4 > 0$, $\CE$ satisfies
	\begin{align*} \label{eq:quadratic_growth_condition_car}
		\CE(v) - \minobj \geq C_3\N{v}_2^2,\quad \textrm{for all } \N{v}_2 \geq C_4,
	\end{align*}
	then there exists a law $\rho \in \CC([0,T], \CP_4(\bbR^d))$ weakly satisfying Equation~\eqref{eq:fokker_planck_anisotropic}.
\end{theorem}

\begin{proof}
	The proof is based on the Leray-Schauder fixed point theorem and uses the same arguments as the ones provided for \cite[Theorems~3.1, 3.2]{carrillo2018analytical}.
\end{proof}

\begin{remark} \label{rem:test_functions_redefine}
As discussed in \cite[Remark~7]{fornasier2021consensus}, the proof of Theorem~\ref{thm:well-posedness_FP_anisotropic} justifies an extension of the test function space $\CC^{\infty}_{c}(\bbR^d)$ in Definition~\ref{def:fokker_planck_anisotropic_weak_sense} to
\begin{equation*}
\begin{split}
	\CC^2_*(\bbR^d) := \big\{\phi\in\CC^2(\bbR^d) :
		&\abs{\partial_k\phi(v)}\leq c(1+\abs{v_k}) \text{ and } \Nnormal{\partial^2_{kk} \phi}_\infty < \infty \\
		&\text{ for all } k\in\{1,\dots,d\} \text{ and some constant } c>0\big\}.
\end{split}
\end{equation*}
\end{remark}

\subsection{Main Results} \label{subsec:main_results}

We now present the main result about global convergence in mean-field law for objective functions that satisfy the following conditions.

\begin{definition}[Assumptions] \label{def:assumptions}
	We consider functions $\CE \in \CC(\bbR^d)$, for which
	\begin{enumerate}[label=A\arabic*,labelsep=10pt,leftmargin=35pt]
		\item\label{asm:zero_global_separable} there exists $v^*\in\bbR^d$ such that $\CE(\globmin)=\inf_{v\in\bbR^d} \CE(v)=:\underbar{\CE}$, and
		\item\label{asm:icp_separable} there exist $\CE_{\infty},R_0,\eta > 0$, and  $\nu\in(0,\infty)$ such that
		\begin{align}
			\Nnormal{v-v^*}_\infty &\leq \frac{1}{\eta}\big(\CE(v)-\underbar{\CE}\big)^{\nu} \quad \textrm{ for all } v \in B^\infty_{R_0}(\globmin), \label{eq:asm_icp_vstar_separable} \\
			\CE_{\infty} &< \CE(v)-\underbar{\CE} \quad \textrm{ for all } v \in \big(B^\infty_{R_0}(\globmin)\big)^c. \label{eq:asm_icp_farfield_separable}
		\end{align}
\end{enumerate}
\end{definition}

Assumption~\ref{asm:icp_separable} can be regarded as a tractability condition of the energy landscape around the minimizer and in the farfield.
Equation~\eqref{eq:asm_icp_vstar_separable} requires the local coercivity of~$\CE$, whereas \eqref{eq:asm_icp_farfield_separable} prevents that $\CE(v)\approx\underbar{\CE}$ far away from~$\globmin$.

Definition~\ref{def:assumptions} covers a wide range of function classes, including for instance the Rastrigin function, see Figure~\ref{fig:objective1d}, and objectives related to various machine learning tasks, see, e.g.,~\cite{fornasier2020consensus_sphere_convergence}.


\begin{theorem} \label{thm:global_convergence_main_anisotropic}
	Let $\CE$ be as in Definition~\ref{def:assumptions}.
	Moreover, let $\rho_0 \in \CP_{\!4}(\bbR^d)$ be such that
	\begin{align} \label{eq:condition_initial_measure}
		\rho_0(B^\infty_{r}(\globmin)) > 0\quad \textrm{for all}\quad r > 0.
	\end{align}
	Define $\CV(\rho_t) := 1/2 \int \N{v-\globmin}_2^2 d\rho_t(v)$.
	Fix any $\varepsilon \in (0,\CV(\rho_0))$ and $\tau \in (0,1)$, parameters $\lambda,\sigma > 0$ with $2\lambda > \sigma^2 $, and the time horizon
	\begin{align} \label{eq:end_time_star_statement_no_H}
		T^* := \frac{1}{(1-\tau)\left(2\lambda-\sigma^2\right)}\log\left(\frac{\CV(\rho_0)}{\varepsilon}\right).
	\end{align}
	Then there exists $\alpha_0>0$, which depends (among problem dependent quantities) on $\varepsilon$ and $\tau$, such that for all $\alpha>\alpha_0$,
	if $\rho \in \CC([0,T^*], \CP_4(\bbR^d))$ is a weak solution to the Fokker-Planck equation~\eqref{eq:fokker_planck_anisotropic} on the time interval $[0,T^*]$ with initial condition $\rho_0$, we have $\min_{t \in [0,T^*]}\CV(\rho_t) \leq \varepsilon$.
	Furthermore, until $\CV(\rho_t)$ reaches the prescribed accuracy $\varepsilon$, we have the exponential decay
	\begin{equation*}
		\CV(\rho_t) \leq \CV(\rho_0) \exp\left(-(1-\tau)\left(2\lambda- \sigma^2\right) t\right)
	\end{equation*}
	and, up to a constant, the same behavior for $W_2^2(\rho_t,\delta_{\globmin})$.
\end{theorem}

The rate of convergence~$(2\lambda-\sigma^2)$ obtained in Theorem~\ref{thm:global_convergence_main_anisotropic} is confirmed numerically by the experiments depicted in Figure~\ref{fig:Vindifferentd}.
We emphasize the dimension-independent convergence of CBO with anisotropic diffusion, contrasting the dimension-dependent rate~$(2\lambda-d\sigma^2)$ of isotropic CBO, cf.\@~\cite[Theorem~12]{fornasier2021consensus}.

Refining the argument of the proof of Theorem~\ref{thm:global_convergence_main_anisotropic} allows to show that the time~$T$, where $\CV(\rho_T)=\varepsilon$ is achieved, satisfies $T\in\big[\frac{1-\tau}{(1+\tau/2)}\;\!T^*,T^*\big]$.
For the technical details we refer to the proof of \cite[Theorem~12]{fornasier2021consensus},
where an analogous statement is shown to hold in the setting of isotropic noise.
A proof for the herein investigated anisotropic setting is provided in full detail in the dissertation of the third author.

\section{Proof of Theorem \ref{thm:global_convergence_main_anisotropic}} \label{sec:proof_main_theorem_anisotropic}

This section provides the proof details for Theorem~\ref{thm:global_convergence_main_anisotropic}, starting with a sketch in Section~\ref{subsec:proof_sketch}.
Sections~\ref{subsec:evolution_convex}--\ref{subsec:lower_bound_prob_mass} present statements, which are needed in the proof and may be of independent interest. Section~\ref{subsec:proof_main} completes the proof.
\begin{remark} \label{rem:wlog_minobj_zero}
	Without loss of generality we assume $\underbar{\CE} = 0$ throughout the proof.
\end{remark}

\subsection{Proof Sketch} \label{subsec:proof_sketch}

The main idea is to show that $\CV(\rho_t)$ satisfies the  differential inequality
\begin{align} \label{eq:proof_sketch_ODE_ineq_anisotropic}
	\frac{d}{dt}\CV(\rho_t) \leq -(1-\tau)\left(2\lambda-\sigma^2\right)\CV(\rho_t)
\end{align}
until $\CV(\rho_T) \leq \varepsilon$.
The first step towards \eqref{eq:proof_sketch_ODE_ineq_anisotropic} is to derive a  differential inequality for $\CV(\rho_t)$ using the dynamics of $\rho$, which is done in Lemma~\ref{lem:evolution_of_objective_anisotropic}.
In order to control the appearing quantity $\N{\conspoint{\rho_t}-v^*}_2$, we establish a quantitative Laplace principle.
Namely, under the inverse continuity property~\ref{asm:icp_separable}, Proposition~\ref{prop:laplace_alt_anisotropic} shows
\begin{align*}
	\N{\conspoint{\rho_t}-v^*}_2\lesssim \ell(r) + \frac{\sqrt{d}\exp(-\alpha r)}{\rho_t(B^\infty_{r}(v^*))},\quad \textrm{for sufficiently small } r > 0,
\end{align*}
where $\ell$ is decreasing with $\ell(r)\rightarrow 0^+$ as $r\rightarrow 0$.
Thus, $\N{\conspoint{\rho_t}-v^*}_2$ can be made arbitrarily small by suitable choices of~$r\ll 1$ and $\alpha \gg 1$, as long as we can guarantee $\rho_t(B^\infty_{r}(v^*)) > 0$ for all $r>0$ and at all times~$t\in[0,T]$.
The latter requires non-zero initial mass $\rho_0(B^\infty_{r}(v^*))$ as well as an active Brownian motion, as made rigorous in Proposition~\ref{lem:lower_bound_probability_anisotropic}.

\subsection{Evolution of the Mean-Field Limit} \label{subsec:evolution_convex}

We now derive the evolution inequality of the energy functional $\CV(\rho_t)$.

\begin{lemma} \label{lem:evolution_of_objective_anisotropic}
	Let $\CE:\bbR^d\rightarrow\bbR$, and fix $\alpha,\lambda,\sigma > 0$.
	Moreover, let $T>0$ and let $\rho \in \CC([0,T], \CP_4(\bbR^d))$ be a weak solution to Equation~\eqref{eq:fokker_planck_anisotropic}.
	Then $\CV(\rho_t)$ satisfies
	\begin{align*}
		\frac{d}{dt}\CV(\rho_t) \leq
		&\,-\left(2\lambda-\sigma^2\right) \CV(\rho_t) + \sqrt{2}\left(\lambda+\sigma^2\right) \sqrt{\CV(\rho_t)} \N{\conspoint{\rho_t}-v^*}_2 \\
		&\,+ \frac{\sigma^2}{2} \N{\conspoint{\rho_t}-v^*}_2^2.
	\end{align*}
\end{lemma}

\begin{proof}
Noting that $\phi(v) = 1/2\N{v-v^*}_2^2$ is in $\CC_*^2(\bbR^d)$ and recalling that $\rho$ satisfies the identity~\eqref{eq:weak_solution_identity_anisotropic} for all test functions in $\CC_*^2(\bbR^d)$, see Remark~\ref{rem:test_functions_redefine},
we obtain
\begin{align*}
	\frac{d}{dt} \CV(\rho_t)
	= -\lambda \int \langle  v-v^*, v-\conspoint{\rho_t}\rangle \,d\rho_t(v) + \frac{\sigma^2}{2} \int \N{v-\conspoint{\rho_t}}_2^2 d\rho_t(v),
\end{align*}
where we used $\partial_k \phi(v) = \left(v-v^*\right)_k$ and $\partial^2_{kk} \phi(v) = 1$ for all $k\in\{1,\dots,d\}$.
Following the steps taken in~\cite[Lemma~17]{fornasier2021consensus} yields the statement. \hfill$\square$
\end{proof}

\subsection{Quantitative Laplace Principle} \label{subsec:quant_laplace}
The Laplace principle asserts that $-\log(\Nnormal{\omegaa}_{L_1(\indivmeasure)})/\alpha \rightarrow \underbar{\CE}$ as $\alpha\rightarrow \infty$ as long as the global minimizer $v^*$ is in the support of $\indivmeasure$.
Under the assumption of the inverse continuity property this can be used to qualitatively characterize the proximity of $\conspoint{\indivmeasure}$ to the global minimizer $v^*$.
However, as it neither allows to quantify this proximity nor gives a suggestion on how to choose $\alpha$ to reach a certain approximation quality, we introduced a quantitative version in \cite[Proposition~21]{fornasier2021consensus}, which we now adapt suitably to satisfy the anisotropic setting.

\begin{proposition} \label{prop:laplace_alt_anisotropic}
	Let $\underbar{\CE} = 0$, $\indivmeasure \in \CP(\bbR^d)$ and fix $\alpha  > 0$. For any $r > 0$ we define $\CE_{r} := \sup_{v \in B^\infty_{r}(\globmin)}\CE(v)$.
	Then, under the inverse continuity property~\ref{asm:icp_separable}, for any $r \in (0,R_0]$ and $q > 0$  such that $q + \CE_{r} \leq \CE_{\infty}$, we have
	\begin{align*}
		\N{\conspoint{\indivmeasure} - v^*}_2 \leq \frac{\sqrt{d}(q + \CE_{r})^{\nu}}{\eta} + \frac{\sqrt{d}\exp(-\alpha q)}{\indivmeasure(B^\infty_{r}(\globmin))}\int\N{v-v^*}_2 d\indivmeasure(v).
	\end{align*}
\end{proposition}

\begin{proof}
Following the lines of the proof of \cite[Proposition~22]{fornasier2021consensus} but replacing all $\ell^2$ balls and norms by $\ell^\infty$ balls and norms, respectively, we obtain
\begin{align*}
	\N{\conspoint{\indivmeasure} - v^*}_\infty \leq \frac{(q+\CE_r)^{\nu}}{\eta} + \frac{\exp\left(-\alpha q\right)}{\indivmeasure(B^\infty_{r}(v^*))}\int\N{v-v^*}_\infty d\indivmeasure(v).
\end{align*}
The statement now follows noting that $\N{\cdot}_\infty\leq\N{\cdot}_2\leq\sqrt{d}\N{\cdot}_\infty$. \hfill$\square$
\end{proof}

\subsection{A Lower Bound for the Probability Mass  around $v^*$} \label{subsec:lower_bound_prob_mass}
In this section we provide a lower bound for the probability mass of $\rho_t(B^\infty_{r}(v^*))$, where $r > 0$ is a small radius.
This is achieved by defining a mollifier $\phi_r$ so that $\rho_t(B^\infty_{r}(v^*)) \geq \int \phi_r(v)\,d\rho_t(v)$ and studying the evolution of the right-hand side.

\begin{lemma} \label{lem:properties_mollifier}
For $r > 0$ we define the mollifier $\phi_r : \bbR^d \rightarrow \bbR$ by
\begin{align} \label{eq:mollifier}
\phi_{r}(v) := \begin{cases}
	\prod_{k=1}^d \exp\left(1-\frac{r^2}{r^2-\left(v-v^*\right)_k^2}\right),& \textrm{ if } \N{v-v^*}_\infty < r,\\
	0,& \textrm{ else.}
\end{cases}
\end{align}
We have $\Im(\phi_r) = [0,1]$, $\supp(\phi_r) = B^\infty_{r}(v^*)$, $\phi_{r} \in \CC_c^{\infty}(\bbR^d)$ and
\begin{align*}
	\partial_{k} \phi_{r}(v) &= -2r^2 \frac{\left(v-v^*\right)_k}{\left(r^2-\left(v-v^*\right)^2_k\right)^2}\phi_{r}(v),\\
	\partial^2_{kk} \phi_{r}(v) &= 2r^2 \left(\frac{2\left(2\left(v-v^*\right)^2_k-r^2\right)\left(v-v^*\right)_k^2 - \left(r^2-\left(v-v^*\right)^2_k\right)^2}{\left(r^2-\left(v-v^*\right)^2_k\right)^4}\right)\phi_{r}(v).
\end{align*}
\end{lemma}

\begin{proof}
	$\phi_r$ is a tensor product of classical well-studied mollifiers. \hfill$\square$
\end{proof}

\begin{proposition} \label{lem:lower_bound_probability_anisotropic}
	Let $T > 0$, $r > 0$, and fix parameters $\alpha,\lambda,\sigma > 0$.
	Assume $\rho\in\CC([0,T],\CP(\bbR^d))$ weakly solves the Fokker-Planck equation~\eqref{eq:fokker_planck_anisotropic} in the sense of
	Definition~\ref{def:fokker_planck_anisotropic_weak_sense} with initial condition $\rho_0 \in \CP(\bbR^d)$ and for $t \in [0,T]$.
	Then, for all $t\in[0,T]$ we have
	\begin{align} \label{eq:lower_bound_probability_anisotropic_rate}
		\rho_t\left(B^\infty_{r}(v^*)\right) 
		&\geq \left(\int\phi_{r}(v)\,d\rho_0(v)\right)\exp\left(-pt\right)
	\end{align}
	with
	\begin{align} \label{eq:def_p_anisotropic}
		p &:= 2d\max\left\{\frac{\lambda(cr\!+\!B\sqrt{c})}{(1\!-\!c)^2r}\!+\!\frac{\sigma^2(cr^2\!+\!B^2)(2c\!+\!1)}{(1\!-\!c)^4r^2},\frac{2\lambda^2}{(2c\!-\!1)\sigma^2}\right\}\!,
	\end{align}
	for any $B<\infty$ with $\sup_{t \in [0,T]}\N{\conspoint{\rho_{t}}-v^*}_\infty \leq B$
	and for any $c \in (1/2,1)$ satisfying $(1-c)^2 \leq (2c-1)c$.
\end{proposition}

\begin{remark}
	In order to ensure a finite decay rate~$p<\infty$ in Proposition~\ref{lem:lower_bound_probability_anisotropic} it is crucial to have a non-vanishing diffusion~$\sigma>0$.
\end{remark}

\begin{proof}[Proposition~\ref{lem:lower_bound_probability_anisotropic}]
By the properties of the mollifier in Lemma~\ref{lem:properties_mollifier} we have
\begin{align*}
	\rho_t\left(B^\infty_{r}(v^*)\right) \geq \int \phi_{r}(v)\,d\rho_t(v).
\end{align*}
Our strategy is to derive a lower bound for the right-hand side of this inequality.
Using the weak solution property of $\rho$ and the fact that  $\phi_{r}\in \CC^{\infty}_c(\bbR^d)$, we obtain
\begin{align}
	\frac{d}{dt}\int\phi_{r}(v)\,d\rho_t(v) &= \sum_{k=1}^d \int \big(T_{1k}(v) + T_{2k}(v)\big) \,d\rho_t(v) \label{eq:initial_evolution} \\
	\textrm{with}\quad T_{1k}(v) &:= -\lambda\left(v-\conspoint{\rho_t}\right)_k \partial_k\phi_r(v) \nonumber \\
	\textrm{and}\quad T_{2k}(v) &:= \frac{\sigma^2}{2} \left(v-\conspoint{\rho_t}\right)_k^2 \partial^2_{kk} \phi_{r}(v) \nonumber
\end{align}
for $k\in\{1,\dots,d\}$.
We now aim for showing $T_{1k}(v) + T_{2k}(v) \geq -p\phi_r(v)$ uniformly on $\bbR^d$ individually for each $k$ and for $p>0$ as in the statement.
Since the mollifier~$\phi_r$ and its derivatives vanish outside of $\Omega_r := \{v \in \bbR^d : \N{v-v^*}_\infty < r\}$
we restrict our attention to the open $\ell^\infty$-ball $\Omega_r$.
To achieve the lower bound over $\Omega_r$, we introduce for each $k\in\{1,\dots,d\}$ the subsets
\begin{align} \label{eq:region_q}
K_{1k} &:= \left\{v \in \bbR^d : \abs{\left(v-v^*\right)_k} > \sqrt{c}r\right\}
\end{align}
and
\begin{align}
\label{region_c}
\begin{aligned}
	&K_{2k} := \bigg\{v \in \bbR^d : -\lambda\left(v-\conspoint{\rho_t}\right)_k \left(v-v^*\right)_k \left(r^2-\left(v-v^*\right)_k^2\right)^2 \\
	&\qquad\qquad\qquad\qquad\qquad\quad > \tilde{c}r^2\frac{\sigma^2}{2} \left(v-\conspoint{\rho_t}\right)_k^2\left(v-v^*\right)_k^2\bigg\},
\end{aligned}
\end{align}
where $\tilde{c} := 2c-1\in(0,1)$.
For fixed $k$ we now decompose $\Omega_r$ according to
\begin{align} \label{eq:disjoint_sets}
\Omega_r = \left(K_{1k}^c \cap \Omega_r\right) \cup \left(K_{1k} \cap K_{2k}^c \cap \Omega_r\right) \cup \left(K_{1k} \cap K_{2k} \cap \Omega_r\right),
\end{align}
which is illustrated in Figure~\ref{fig:decomposition_support} for different positions of $\conspoint{\rho_t}$ and values of $\sigma$.
\begin{figure}[!ht]
\centering
\subcaptionbox{\footnotesize $\conspoint{\rho_t}\in\Omega_r$, $\sigma=0.2$\label{fig:decomposition1}}{\includegraphics[width=0.29\textwidth, trim=112 244 134 250,clip]{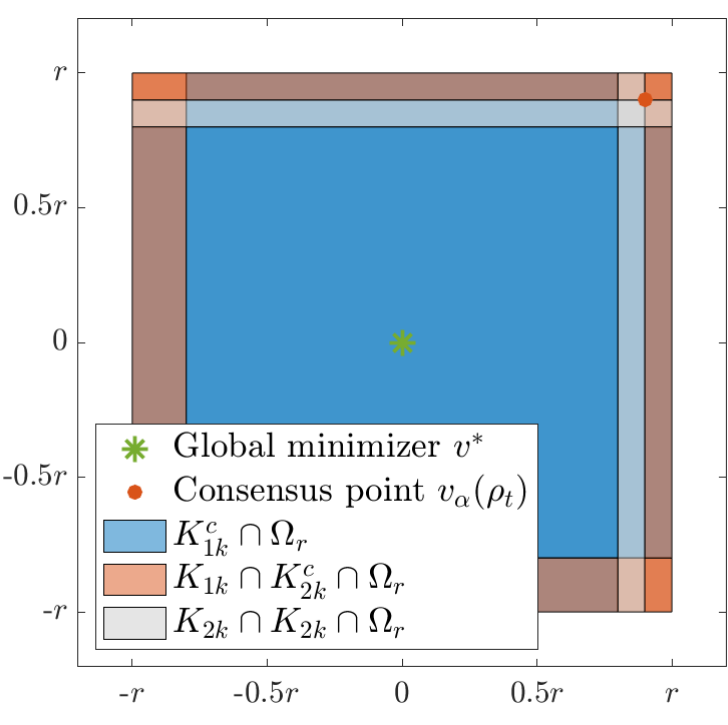}}
\hspace{1.5em}
\subcaptionbox{\footnotesize $\conspoint{\rho_t}\not\in\Omega_r$, $\sigma=0.2$\label{fig:decomposition3}}{\includegraphics[width=0.29\textwidth, keepaspectratio, trim=112 244 134 250,clip]{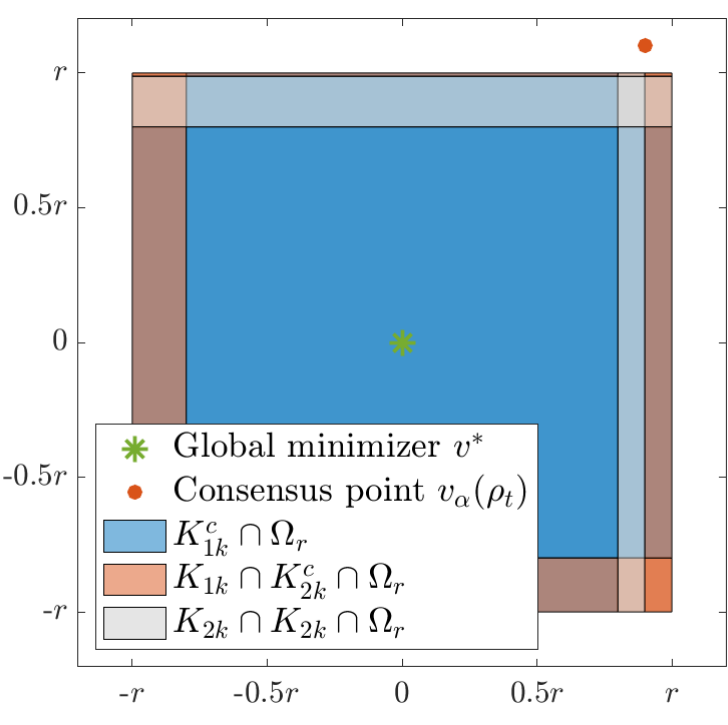}}
\hspace{1.5em}
\subcaptionbox{\footnotesize $\conspoint{\rho_t}\not\in\Omega_r$, $\sigma=1$\label{fig:decomposition2}}{\includegraphics[width=0.29\textwidth, trim=112 244 134 250,clip]{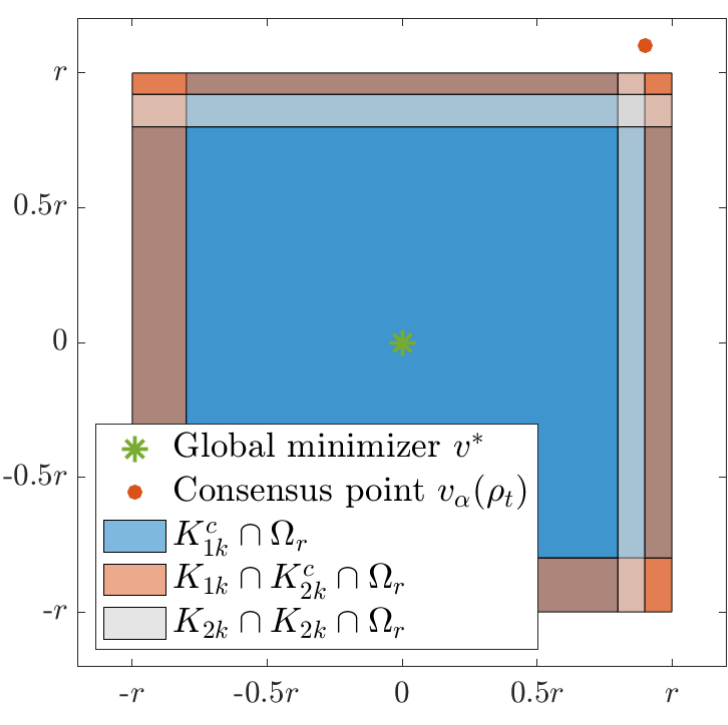}}
\caption{Visualization of the decomposition of $\Omega_r$ as in \eqref{eq:disjoint_sets} for different positions of $\conspoint{\rho_t}$ and values of $\sigma$.}
\label{fig:decomposition_support} \vspace{-1em}
\end{figure}

\noindent
In the following we treat each of these three subsets separately.

\noindent
\textbf{Subset $K_{1k}^c \cap \Omega_r$:}
We have $\abs{\left(v-v^*\right)_k} \leq \sqrt{c}r$ for each $v \in K_{1k}^c$, which can be used to independently derive lower bounds for both terms $T_{1k}$ and $T_{2k}$.
For $T_{1k}$, we insert the expression for $\partial_k \phi_{r}$ from Lemma~\ref{lem:properties_mollifier} to get
\begin{equation*}
\begin{aligned}
	T_{1k}(v)
	&= 2r^2\lambda \left(v-\conspoint{\rho_t}\right)_k \frac{\left(v-v^*\right)_k}{\left(r^2-\left(v-v^*\right)^2_k\right)^2}\phi_{r}(v)\\
	&\geq -2 r^2 \!\lambda \frac{\abs{\left(v-\conspoint{\rho_t}\right)_k}\!\abs{\left(v-v^*\right)_k}}{\left(r^2-\left(v-v^*\right)^2_k\right)^2}\phi_r(v)
	\geq - \frac{2\lambda(\sqrt{c}r+B)\sqrt{c}}{(1-c)^2r}\phi_r(v) \\
	&=: -p_1\phi_r(v),
\end{aligned}
\end{equation*}
where $\abs{(v-\conspoint{\rho_t})_k} \leq \abs{(v-v^*)_k}+\abs{(v^*-\conspoint{\rho_t})_k} \leq \sqrt{c}r+B$ is used in the last inequality.
For $T_2$ we insert the expression for $\partial^2_{kk} \phi_{r}$ from Lemma~\ref{lem:properties_mollifier} to obtain
\begin{equation*}
\begin{aligned}
	T_{2k}(v)
	&= \sigma^2r^2 \left(v\!-\!\conspoint{\rho_t}\right)_k^2 \frac{2\left(2\left(v\!-\!v^*\right)^2_k\!-\!r^2\right)\left(v\!-\!v^*\right)_k^2 \!-\! \left(r^2\!-\!\left(v\!-\!v^*\right)^2_k\right)^2}{\left(r^2\!-\!\left(v\!-\!v^*\right)^2_k\right)^4}\phi_{r}(v)\\
	&\geq - \frac{2\sigma^2(cr^2\!+\!B^2)(2c\!+\!1)}{(1\!-\!c)^4r^2}\phi_{r}(v) =: -p_2\phi_r(v),
\end{aligned}
\end{equation*}
where the last inequality uses $(v-\conspoint{\rho_t})_k^2 \leq \left(\sqrt{c}r+B\right)^2 \leq 2(cr^2+B^2)$.

\noindent
\textbf{Subset $K_{1k} \cap K_{2k}^c \cap \Omega_r$:}
As $v \in K_{1k}$ we have $\abs{\left(v-v^*\right)_k}_2 > \sqrt{c} r$.
We observe that $T_{1k}(v)+T_{2k}(v) \geq 0$ for all $v$ in this subset whenever
\begin{equation} \label{eq:aux_term_3}
\begin{aligned}
	&\left(-\lambda \left(v-\conspoint{\rho_t}\right)_k\left(v-v^*\right)_k + \frac{\sigma^2}{2} \left(v-\conspoint{\rho_t}\right)_k^2\right)\left(r^2-\left(v-v^*\right)^2_k\right)^2 \\
	&\qquad\qquad\qquad\qquad\qquad\, \leq \sigma^2 \left(v-\conspoint{\rho_t}\right)_k^2 \left(2\left(v-v^*\right)^2_k-r^2\right)\left(v-v^*\right)_k^2.
\end{aligned}
\end{equation}
The first term on the left-hand side in \eqref{eq:aux_term_3} can be bounded from above exploiting that $v\in K_{2k}^c$ and by using the relation $\tilde{c} = 2c-1$.
More precisely, we have
\begin{align*}
	&\!-\!\lambda\! \left(v\!-\!\conspoint{\rho_t}\right)_k\!\left(v\!-\!v^*\right)_k \!\left(r^2\!-\!\left(v\!-\!v^*\right)^2_k\right)^2\!\!
	\leq \tilde{c}r^2\frac{\sigma^2}{2}\!\left(v\!-\!\conspoint{\rho_t}\right)_k^2\!\left(v\!-\!v^*\right)_k^2\\
	&\ \ =\! (2c\!-\!1)r^2\frac{\sigma^2}{2}\!\left(v\!-\!\conspoint{\rho_t}\right)_k^2\!\left(v\!-\!v^*\right)_k^2\!
	\leq\! \left(2\left(v\!-\!v^*\right)_k^2\!-\!r^2\right)\!\frac{\sigma^2}{2}\!\left(v\!-\!\conspoint{\rho_t}\right)_k^2\!\left(v\!-\!v^*\right)_k^2\!,
\end{align*}
where the last inequality follows since $v\in K_{1k}$.
For the second term on the left-hand side in \eqref{eq:aux_term_3} we can use $(1-c)^2 \leq (2c-1)c$ as per assumption, to get
\begin{align*}
	&\frac{\sigma^2}{2} \left(v\!-\!\conspoint{\rho_t}\right)_k^2 \left(r^2\!-\!\left(v\!-\!v^*\right)_k^2\right)^2
	\leq \frac{\sigma^2}{2} \left(v\!-\!\conspoint{\rho_t}\right)_k^2 (1\!-\!c)^2r^4 \\
	&\,\quad\leq \frac{\sigma^2}{2}\! \left(v\!-\!\conspoint{\rho_t}\right)_k^2 (2c\!-\!1)r^2cr^2
	\leq \frac{\sigma^2}{2}\! \left(v\!-\!\conspoint{\rho_t}\right)_k^2 \left(2\left(v\!-\!v^*\right)_k^2\!-\!r^2\right)\left(v\!-\!v^*\right)_k^2.
\end{align*}
Hence, \eqref{eq:aux_term_3} holds and we have $T_{1k}(v) + T_{2k}(v) \geq 0$ uniformly on this subset.

\noindent
\textbf{Subset $K_{1k} \cap K_{2k} \cap \Omega_r$:}
As $v \in K_{1k}$ we have $\abs{\left(v-v^*\right)_k}_2 > \sqrt{c} r$.
We first note that $T_{1k}(v) = 0$ whenever $\sigma^2\left(v-\conspoint{\rho_t}\right)_k^2 = 0$, provided $\sigma>0$, so nothing needs to be done if $v_k = \left(\conspoint{\rho_t}\right)_k$.
Otherwise, if $\sigma^2\left(v-\conspoint{\rho_t}\right)_k^2 > 0$, we exploit $v\in K_{2k}$ to get
\begin{equation*}
\begin{split}
	\frac{\left(v-\conspoint{\rho_t}\right)_k\left(v-v^*\right)_k}{\left(r^2-\left(v-v^*\right)_k^2\right)^2}
	&\geq\frac{-\abs{\left(v-\conspoint{\rho_t}\right)_k}\abs{\left(v-v^*\right)_k}}{\left(r^2-\left(v-v^*\right)_k^2\right)^2}\\
	&> \frac{2\lambda\left(v-\conspoint{\rho_t}\right)_k \left(v-v^*\right)_k}{\tilde{c}r^2\sigma^2\abs{\left(v-\conspoint{\rho_t}\right)_k}\abs{\left(v-v^*\right)_k}}
	\geq -\frac{2\lambda}{\tilde{c}r^2\sigma^2}.
\end{split}
\end{equation*}
Using this, $T_{1k}$ can be bounded from below by
\begin{align*}
	T_{1k}(v) &= 2r^2\lambda \left(v-\conspoint{\rho_t}\right)_k \frac{\left(v-v^*\right)_k}{\left(r^2-\left(v-v^*\right)^2_k\right)^2}\phi_{r}(v) \geq -\frac{4\lambda^2}{\tilde{c}\sigma^2}\phi_{r}(v) =: -p_3\phi_r(v).
\end{align*}
For $T_{2k}$, the nonnegativity of $\sigma^2\left(v-\conspoint{\rho_t}\right)_k^2$ implies $T_{2k}(v) \geq 0$, whenever
\begin{equation*}
	2\left(2\left(v-v^*\right)^2_k-r^2\right)\left(v-v^*\right)_k^2 \geq \left(r^2-\left(v-v^*\right)^2_k\right)^2.
\end{equation*}
This holds for $v\in K_{1k}$, if $2(2c-1)c \geq (1-c)^2$ as implied by the assumption.

\noindent
\textbf{Concluding the proof:}
Using the evolution of $\phi_r$ as in \eqref{eq:initial_evolution} and the individual decompositions of $\Omega_r$ for the terms $T_{1k}+T_{2k}$, we now get
\begin{equation*}
\begin{split}
	&\frac{d}{dt}\int\!\phi_{r}(v)\,d\rho_t(v)
	= \sum_{k=1}^d\!\Bigg(\!\int_{K_{1k} \cap K_{2k}^c \cap \Omega_r}\!\underbrace{(T_{1k}(v) + T_{2k}(v))}_{\geq 0}\,d\rho_t(v)\\
	&\qquad+\! \int_{K_{1k} \cap K_{2k} \cap \Omega_r}\!\underbrace{(T_{1k}(v) + T_{2k}(v))}_{\geq -p_3\phi_{r}(v)}\,d\rho_t(v)
	+\!\int_{K_{1k}^c \cap \Omega_r}\!\underbrace{(T_{1k}(v) + T_{2k}(v))}_{\geq -(p_1+p_2)\phi_{r}(v)}d\rho_t(v)\!\Bigg) \\
	&\quad\geq -d\max\left\{p_1+p_2,p_3\right\} \int\phi_{r}(v)\,d\rho_t(v) = -p \int\phi_{r}(v)\,d\rho_t(v).
\end{split}
\end{equation*}
An application of Gr\"onwall's inequality concludes the proof. \hfill$\square$
\end{proof}


\subsection{Proof of Theorem~\ref{thm:global_convergence_main_anisotropic}} \label{subsec:proof_main}
We now have all necessary tools to conclude the global convergence proof.
\begin{proof}[Theorem~\ref{thm:global_convergence_main_anisotropic}]
	Let us first choose the parameter~$\alpha$ such that
	\begin{align} \label{eq:alpha}
	\begin{split}
		\alpha>
		\alpha_0
		:= \frac{1}{q_\varepsilon}\Bigg(\log\left(\frac{2^{d+1}\sqrt{2d\CV(\rho_0)}}{c\left(\tau,\lambda,\sigma\right)\sqrt{\varepsilon}}\right) 
		&+ \frac{p}{(1-\tau)\left(2\lambda-\sigma^2\right)}\log\left(\frac{\CV(\rho_0)}{\varepsilon}\right) \\
		&- \log\rho_0\big(B_{r_\varepsilon/2}^\infty(v^*)\big)\!\Bigg),
	\end{split}
	\end{align}
	where we introduce the definitions
	\begin{align}
		c\left(\tau,\lambda,\sigma\right)
		:= \min\left\{
			\frac{\tau}{2}\frac{\left(2\lambda-\sigma^2\right)}{\sqrt{2}\left(\lambda+\sigma^2\right)}, 
			\sqrt{\tau\frac{\left(2\lambda-\sigma^2\right)}{\sigma^2}}
			\right\} 
	\end{align}
	as well as
	\begin{align*} 
		q_\varepsilon \!:= \frac{1}{2}\min\bigg\{\!\left(\eta\frac{c\left(\tau,\lambda,\sigma\right)\sqrt{\varepsilon}}{2\sqrt{d}}\right)^{1/\nu}\!\!,\CE_{\infty}\!\bigg\}
		\text{ and }
		r_\varepsilon \!:=\!\max_{s \in [0,R_0]}\left\{\max_{v \in B_s^\infty(\globmin)}\CE(v) \leq q_\varepsilon\right\}\!.
	\end{align*}
	Moreover, $p$ is as defined in \eqref{eq:def_p_anisotropic} in Proposition~\ref{lem:lower_bound_probability_anisotropic} with $B=c(\tau,\lambda,\sigma)\sqrt{\CV(\rho_0)}$ and with $r=r_\varepsilon$.
	By construction, $q_\varepsilon>0$ and $r_\varepsilon\leq R_0$.
	Moreover, recalling the notation $\CE_{r}=\sup_{v \in B_{r}^\infty(\globmin)}\CE(v)$ from Proposition~\ref{prop:laplace_alt_anisotropic}, we have $q_\varepsilon+\CE_{r_\varepsilon} \leq 2q_\varepsilon \leq \CE_{\infty}$ by definition of~$r_\varepsilon$.
	Furthermore, since $q_\varepsilon>0$, the continuity of $\CE$ ensures that there exists $s_{q_\varepsilon}>0$ such that $\CE(v)\leq q_\varepsilon$ for all $v\in B_{s_{q_\varepsilon}}^\infty(\globmin)$, yielding also $r_\varepsilon>0$.
	
	Let us now define the time horizon $T_\alpha \geq 0$ by
	\begin{align} \label{eq:endtime_T_anisotropic}
		T_\alpha := \sup\big\{t\geq0 : \CV(\rho_{t'}) \!>\! \varepsilon \text{ and } \N{\conspoint{\rho_{t'}}-\globmin}_2 \!<\! C(t') \text{ for all } t' \!\in\! [0,t]\big\}
	\end{align}
	with $C(t):=c(\tau,\lambda,\sigma)\sqrt{\CV(\rho_t)}$.
	Notice for later use that $C(0)=B$.
	Our aim is to show that $\min_{t \in [0,T_\alpha]}\CV(\rho_t) \leq \varepsilon$ with $T_\alpha\leq T^*$ and that we have at least exponential decay until $\CV(\rho_t)$ reaches the prescribed accuracy~$\varepsilon$.
	
	First, however, let us ensure that $T_\alpha>0$, which follows from the continuity of the mappings~$t\mapsto\CV(\rho_{t})$ and~$t\mapsto\N{\conspoint{\rho_{t}}-\globmin}_2$ since $\CV(\rho_{0}) > \varepsilon$ and $\N{\conspoint{\rho_{0}}-\globmin}_2 < C(0)$.
	While the former holds by assumption, the latter follows by applying Proposition~\ref{prop:laplace_alt_anisotropic} with $q_\varepsilon$ and $r_\varepsilon$ as defined before, which gives
	\begin{align*} 
		\N{\conspoint{\rho_{0}} - \globmin}_2
		&\leq \frac{\sqrt{d}\left(q_\varepsilon+\CE_{r_\varepsilon}\right)^\nu}{\eta} + \frac{\sqrt{d}\exp\left(-\alpha q_\varepsilon\right)}{\rho_0(B_{r_\varepsilon}^\infty(v^*))}\int\N{v-\globmin}_2d\rho_{0}(v)\\
		&\leq \frac{c\left(\tau,\lambda,\sigma\right)\sqrt{\varepsilon}}{2} + \frac{\sqrt{d}\exp\left(-\alpha q_\varepsilon\right)}{\rho_0(B_{r_\varepsilon}^\infty(v^*))}\sqrt{2\CV(\rho_0)}\\
		&\leq c\left(\tau,\lambda,\sigma\right)\sqrt{\varepsilon}
		< c\left(\tau,\lambda,\sigma\right)\sqrt{\CV(\rho_0)} = C(0),
	\end{align*}
	where the first inequality in the last line holds by the choice of $\alpha$ in \eqref{eq:alpha}.
	
	Let us now show that $\CV(\rho_t)$ decays at least exponentially fast up to time $T_\alpha$.
	Lemma~\ref{lem:evolution_of_objective_anisotropic} provides an upper bound for the time derivative of $\CV(\rho_t)$, given by
	\begin{equation}
	\begin{aligned} \label{eq:main_aux_1_anisotropic}
		\frac{d}{dt}\CV(\rho_t) \leq\,
		& -\left(2\lambda-\sigma^2\right) \CV(\rho_t) + \sqrt{2}\left(\lambda+\sigma^2\right) \sqrt{\CV(\rho_t)} \N{\conspoint{\rho_t}-\globmin}_2 \\
		&+ \frac{\sigma^2}{2} \N{\conspoint{\rho_t}-\globmin}_2^2.
	\end{aligned}
	\end{equation}
	Combining this with the definition of $T_\alpha$ in \eqref{eq:endtime_T_anisotropic} we have by construction 
	\begin{align*}
		\frac{d}{dt}\CV(\rho_t)
		\leq -(1-\tau)\left(2\lambda-\sigma^2\right)\CV(\rho_t)
		\quad \text{ for all } t \in (0,T_\alpha).
	\end{align*}
	Gr\"onwall's inequality implies the upper bound 
	\begin{align} \label{eq:evolution_J_no_H}
		\CV(\rho_t)
		\leq \CV(\rho_0) \exp\left(- (1-\tau)\left(2\lambda-\sigma^2\right) t\right)
		\quad \text{ for all } t \in [0,T_\alpha].
	\end{align}
	Accordingly, we note that $\CV(\rho_t)$ is decreasing in $t$, which implies the decay of the function $C(t)$ as well. 
	Hence, recalling the definition of $T_\alpha$, we may bound
	\begin{align}
		&\max_{t \in [0,T_\alpha]}\N{\conspoint{\rho_t}-\globmin}_2
		\leq \max_{t \in [0,T_\alpha]} C(t)\leq C(0).
		\label{eq:max_bound_distance_anisotropic}
	\end{align}
	We now conclude by showing $\min_{t \in [0,T_\alpha]}\CV(\rho_t) \leq \varepsilon$ with $T_\alpha\leq T^*$.
	For this we distinguish the following three cases.
	
	\noindent
	\textbf{Case $T_\alpha \geq T^*$:}
	If $T_\alpha \geq T^*$, we can use the definition of $T^*$ in \eqref{eq:end_time_star_statement_no_H} and the time-evolution bound of $\CV(\rho_t)$ in \eqref{eq:evolution_J_no_H} to conclude that $\CV(\rho_{T^*}) \leq \varepsilon$.
	Hence, by definition of $T_\alpha$ in \eqref{eq:endtime_T_anisotropic}, we find $\CV(\rho_{T_\alpha}) =\varepsilon$ and $T_\alpha = T^*$.
	
	\noindent
	\textbf{Case $T_\alpha < T^*$ and $\CV(\rho_{T_\alpha}) \leq \varepsilon$:}
	Nothing needs to be discussed in this case.
	
	\noindent
	\textbf{Case $T_\alpha < T^*$ and $\CV(\rho_{T_\alpha}) > \varepsilon$:}
	We shall show that this case can never occur by verifying that $\N{\conspoint{\rho_{T_\alpha}}-\globmin}_2 < C(T_\alpha)$ due to the choice of $\alpha$ in~\eqref{eq:alpha}.
	Namely, by applying again Proposition~\ref{prop:laplace_alt_anisotropic} with $q_\varepsilon$ and $r_\varepsilon$ as defined before, we get
	\begin{align} \label{eq:proof_contradiction_1}
	\begin{split}
		\N{\conspoint{\rho_{T_\alpha}}-\globmin}_2
		&\leq \frac{\sqrt{d}\left(q_\varepsilon+\CE_{r_\varepsilon}\right)^\nu}{\eta} + \frac{\sqrt{d}\exp\left(-\alpha q_\varepsilon\right)}{\rho_{T_\alpha}\big(B_{r_\varepsilon}^\infty(v^*)\big)}\int\N{v-\globmin}_2d\rho_{T_\alpha}(v)\\
		&\leq \frac{c\left(\tau,\lambda,\sigma\right)\sqrt{\varepsilon}}{2} + \frac{\sqrt{d}\exp\left(-\alpha q_\varepsilon\right)}{\rho_{T_\alpha}\big(B_{r_\varepsilon}^\infty(v^*)\big)}\sqrt{2\CV(\rho_{T_\alpha})}\\
		&< \frac{c\left(\tau,\lambda,\sigma\right)\sqrt{\CV(\rho_{T_\alpha})}}{2} + \frac{\sqrt{d}\exp\left(-\alpha q_\varepsilon\right)}{\rho_{T_\alpha}\big(B_{r_\varepsilon}^\infty(v^*)\big)}\sqrt{2\CV(\rho_{T_\alpha})}.
	\end{split}
	\end{align}
	Since, thanks to \eqref{eq:max_bound_distance_anisotropic}, we have the bound $\max_{t \in [0,T_\alpha]}\Nnormal{\conspoint{\rho_t}-\globmin}_2 \leq B$ for $B=C(0)$, which is in particular independent of $\alpha$, Proposition~\ref{lem:lower_bound_probability_anisotropic} guarantees that there exists a $p>0$ independent of $\alpha$ (but dependent on $B$ and $r_\varepsilon$) with
	\begin{align*}
		\rho_{T_\alpha}(B_{r_\varepsilon}^\infty(v^*))
		&\geq \left(\int \phi_{r_\varepsilon}(v) \,d\rho_0(v)\right)\exp(-pT_\alpha) \\
		&\geq \frac{1}{2^d}\,\rho_0\big(B_{r_\varepsilon/2}^\infty(v^*)\big) \exp(-pT^*) 
		> 0,
	\end{align*}
	where we used \eqref{eq:condition_initial_measure} for bounding the initial mass $\rho_0$ and the fact that $\phi_{r}$ (as defined in Equation~\eqref{eq:mollifier}) is bounded from below on $B_{r/2}^\infty(\globmin)$ by $1/2^d$.
	With this we can continue the chain of inequalities in~\eqref{eq:proof_contradiction_1} to obtain
	\begin{align*} 
	\begin{split}
		\N{\conspoint{\rho_{T_\alpha}}-\globmin}_2
		&< \frac{c\left(\tau,\lambda,\sigma\right)\sqrt{\CV(\rho_{T_\alpha})}}{2} + \frac{2^d\sqrt{d}\exp\left(-\alpha q_\varepsilon\right)}{\rho_0\big(B_{r_\varepsilon/2}^\infty(v^*)\big) \exp(-pT^*)}\sqrt{2\CV(\rho_{T_\alpha})}\\
		&\leq c\left(\tau,\lambda,\sigma\right)\sqrt{\CV(\rho_{T_\alpha})}
		= C(T_\alpha),
	\end{split}
	\end{align*}
	where the first inequality in the last line holds by the choice of $\alpha$ in \eqref{eq:alpha}.
	This establishes the desired contradiction, again as consequence of the continuity of the mappings~$t\mapsto\CV(\rho_{t})$ and~$t\mapsto\N{\conspoint{\rho_{t}}-\globmin}_2$. \hfill$\square$
\end{proof}

\section{A Machine Learning Example} \label{sec:numerics}

In this section, we showcase the practicability of the implementation of anisotropic CBO as described in~\cite[Algorithm~2.1]{carrillo2019consensus} for problems appearing in machine learning by training a shallow and a convolutional NN~(CNN) classifier for the MNIST dataset of handwritten digits~\cite{MNIST}.
Let us emphasize that it is not our aim to challenge the state of the art for this task by employing the most sophisticated model or intricate data preprocessing.
We merely believe that this is a well-understood, complex, high-dimensional benchmark to demonstrate that CBO achieves good results already with limited computational capacities.

Let us now describe the NN architectures used in our numerical experiment, see also Figure~\ref{fig:architectures}. 
\begin{figure}[!ht]
\centering
\subcaptionbox{\label{fig:shallowNN} Shallow NN with one dense layer}{\vspace{1.01em}\includegraphics[width=0.24\textwidth, trim=0 0 0 0,clip]{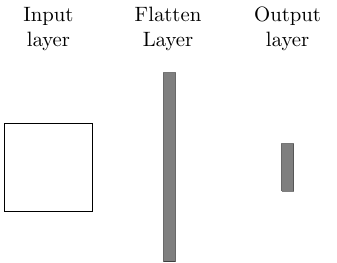}}
\hspace{2em}
\subcaptionbox{\label{fig:CNN} Convolutional NN~(LeNet-1) with two convolutional and two pooling layers, and one dense layer}{\includegraphics[width=0.654\textwidth, trim=0 0 0 0,clip]{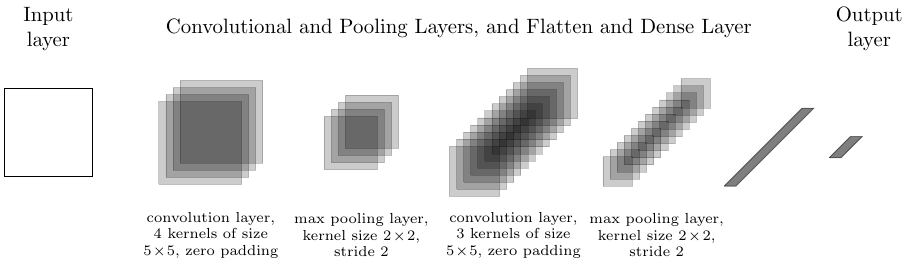}}
\caption{Architectures of the NNs used in the experiments of Section~\ref{sec:numerics}.} \vspace{-1em}
\label{fig:architectures}
\end{figure}
%
Each input image is represented by a matrix of dimension $28\times28$ with entries valued between $0$ and $1$ depending on the grayscale of the respective pixel.
For the shallow neural net (see Figure~\ref{fig:shallowNN}) the image is first reshaped to a vector~$x\in\bbR^{728}$ before being passed through a dense layer of the form~$\relu(Wx+b)$ with trainable weight matrix~$W\in\bbR^{10\times728}$ and bias vector~$b\in\bbR^{10}$.
The CNN (see Figure~\ref{fig:CNN}) has learnable kernels and its architecture is similar to the one of the LeNet-1, cf.\@~\cite[Section~III.C.7]{lecun1998gradient}.
In both networks a batch normalization step is included after each $\relu$ activation, which entails a considerably faster training process.
Moreover, in the final layers a softmax activation function is applied so that the output can be interpreted as a probability distribution over the digits.
In total, the number of unknowns to be trained in case of the shallow NN is~$7850$, which compares to~$2112$ free parameters for the CNN.
We denote the parameters of the NN by~$\theta$ and its forward pass by $f_\theta$.

As a loss function during training we use the categorical crossentropy loss $\ell(\widehat{y},y)=-\sum_{k=0}^9 y_k \log \left(\widehat{y}_k\right)$ with $\widehat{y}=f_\theta(x)$ denoting the output of the NN for a training sample~$(x,y)$ consisting of image and label.
This gives rise to the objective function~$\CE(\theta) = \frac{1}{M} \sum_{m=1}^M \ell(f_\theta(x^m),y^m)$, where $(x^m,y^m)_{m=1}^M$ denote the $M$ training samples.
When evaluating the performance of the NN we determine the accuracy on a test set by counting the number of successful predictions.

The used implementation of anisotropic CBO combines ideas presented in~\cite[Section~2.2]{fornasier2020consensus_sphere_convergence} with the algorithm proposed in~\cite{carrillo2019consensus}.
More precisely, it employs random mini-batch ideas when evaluating the objective function~$\CE$ and when computing the consensus point~$\conspointnoarg$, meaning that $\CE$ is only evaluated on a random subset of size~$n_\CE$ of the training dataset and $\conspointnoarg$ is only computed from a random subset of size~$n_N$ of all particles.
While this reduces the computational complexity, it simultaneously increases the stochasticity, which enhances the ability to escape from local optima.
Furthermore, inspired by Simulated Annealing, a cooling strategy for the parameters~$\alpha$ and $\sigma$ is used as well as a variance-based particle reduction technique similar to ideas from Genetic Algorithms.
More specifically, $\alpha$ is multiplied by $2$ after each epoch, while the diffusion parameter~$\sigma$ follows the schedule $\sigma_{epoch} = \sigma_{0}/\log_2(epoch+2)$.
For our experiments we choose the parameters $\lambda=1$, $\sigma_0=\sqrt{0.4}$ and $\alpha_{initial}=50$, and  discrete time step size $\Delta t=0.1$ for training both the shallow and the convolutional NN.
We use $N=100$ particles, which are initialized according to $\CN\big((0,\dots,0)^T,\Id\big)$.
The mini-batch sizes are $n_\CE=60$ and $n_N=10$ and despite $\conspointnoarg$ being computed only on a basis of $n_N$ particles, all $N$ particles are updated in every time step, referred to as the full update in~\cite{carrillo2019consensus}.
We emphasize that hyperparameters have not been tuned extensively.

In Figure~\ref{fig:MNISTresults} we report the results of our experiment.
\begin{figure}[!ht]
	\centering
	\includegraphics[height=4.2cm, trim=28 248 8 266,clip]{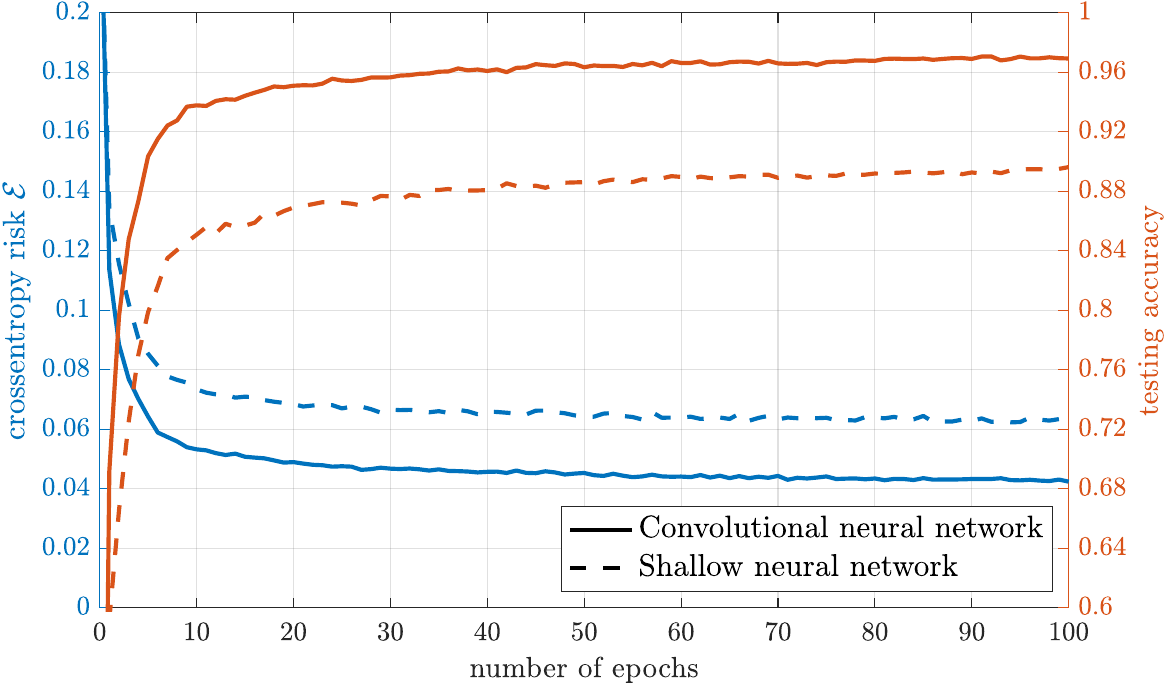}
\caption{Comparison of the performances of a shallow (dashed lines) and convolutional (solid lines) NN with architectures as described in Figures~\ref{fig:shallowNN} and~\ref{fig:CNN}, when trained with a discretized version of the anisotropic CBO dynamics~\eqref{eq:dyn_micro}. Depicted are the accuracies on a test dataset~(orange lines) and the values of the objective function~$\CE$~(blue lines), which was chosen to be the categorical crossentropy loss on a random sample of the training set of size $10000$.} \vspace{-1em}
\label{fig:MNISTresults}
\end{figure}
While achieving a test accuracy of almost $90\%$ for the shallow NN, we obtain around $97\%$ accuracy with the CNN.
For comparison, when trained with backpropagation with finely tuned parameters, a comparable CNN achieves $98.3\%$ accuracy, cf.\@~\cite[Figure~9]{lecun1998gradient}.
In view of these results, CBO can be regarded as a successful optimizer for machine learning tasks, which performs comparably to the state of the art.
At the same time it is worth highlighting that CBO is extremely versatile and customizable, does not require gradient information or substantial hyperparameter tuning and has the potential to be parallelized.

\section{Conclusion} \label{sec:conclusion}
In this paper we establish the global convergence of anisotropic consensus-based optimization (CBO) to the global minimizer in mean-field law with dimension-independent convergence rate by adapting the proof technique developed in~\cite{fornasier2021consensus}.
It is based on the insight that the dynamics of individual particles follow, on average, the gradient flow dynamics of the map $v\mapsto\N{v-v^*}^2_2$.
Furthermore, by utilizing the implementation of anisotropic CBO suggested in~\cite{carrillo2019consensus}, we demonstrate the practicability of the method by training the well-known LeNet-1 on the MNIST data set, achieving around $97$\% accuracy after few epochs with just 100 particles.

In subsequent work we plan to extend our theoretical understanding of CBO to the finite particle regime, and aim to provide extensive numerical studies.
We also intend to use this approach to explain the mean-field law convergence behavior of other metaheuristics such as Particle Swarm Optimization, see, e.g., \cite{cipriani2021zeroinertia,qiu2022globalPSOconvergence}.

\subsubsection{Acknowledgements.}
MF acknowledges the support of the DFG Project “Identification of Energies from Observations of Evolutions” and the DFG SPP 1962 “Non-smooth and Complementarity-Based Distributed Parameter Systems: Simulation and Hierarchical Optimization”.
KR acknowledges the financial support from the Technical University of Munich -- Institute for Ethics in Artificial Intelligence (IEAI).

\bibliographystyle{splncs04}
\bibliography{biblio.bib}

\end{document}